\documentclass[a4paper, 12pt]{article}
\usepackage{amssymb}
\usepackage{amsmath}
\usepackage{amsthm}
\usepackage[dvipdfm]{graphicx}
\usepackage{verbatim,enumerate}
\usepackage[active]{srcltx}
 \newtheorem{theorem}{Theorem}[section]
 \newtheorem*{theorem*}{Theorem}
 \newtheorem*{lemma*}{Lemma}
 \newtheorem{proposition}[theorem]{Proposition}
 \newtheorem{fact}[theorem]{Fact}
 \newtheorem{fact*}{Fact}
 \newtheorem{lemma}[theorem]{Lemma}
 
\theoremstyle{definition}

 \newtheorem*{remark*}{Remark}

\numberwithin{equation}{section}

\usepackage[usenames]{color}
\newcommand{\red}[1]{\textcolor{Black}{#1}}

\newcommand{\R}{\boldsymbol{R}}

\newcommand{\rank}{\operatorname{rank}}

\renewcommand{\phi}{\varphi}
\newcommand{\sgn}{\operatorname{sgn}}
\newcommand{\inner}[2]{\left\langle{#1},{#2}\right\rangle}

\newcommand{\A}{\mathcal{A}}

\newcommand{\ep}{\varepsilon}

\newcommand{\pmt}[1]{{\begin{pmatrix} #1  \end{pmatrix}}}

\newcommand{\mycomment}[1]{}

\usepackage{array}
\newcolumntype{L}{>{\displaystyle}l}
\newcolumntype{C}{>{\displaystyle}c}
\newcolumntype{R}{>{\displaystyle}r}
\makeatletter
  \newcommand{\subsubsubsection}{\@startsection{paragraph}{4}{\z@}%
    {-1ex \@plus -1ex \@minus -.2ex}%
    {1.0ex \@plus.2ex}
    {\reset@font\bfseries\normalsize}
  }
  \makeatother
\begin{document}
\begin{center}
{\large {\bf Geometry of cuspidal edges with boundary}}
\\[2mm]
\today
\\[2mm]
\renewcommand{\thefootnote}{\fnsymbol{footnote}}
Luciana F. Martins
 and
Kentaro Saji
\footnote[0]{
Partly supported by the
Japan Society for the Promotion of Science (JSPS)
and the
Coordenadoria de Aperfei\c{c}oamento de Pessoal de N\'ivel Superior
under the Japan-Brazil research cooperative
program and the
JSPS KAKENHI Grant Number 26400087.
}

\begin{quote}
{\small We study differential geometric properties of cuspidal edges
with boundary. There are several differential geometric invariants
which are related with the behavior of the boundary in addition to
usual differential geometric invariants of cuspidal edges. We study
the relation of these invariants with several other invariants.}
\end{quote}
\end{center}
\section{Maps from manifolds with boundary}
There are several studies for $C^\infty$ map-germs $f:(\R^m,0)\to
(\R^n,0)$ with $\A$-equivalence. Two map-germs $f,g:(\R^m,0)\to
(\R^n,0)$ are {\em ${\cal A}$-equivalent\/} if there exist
diffeomorphisms $\phi:(\R^m,0)\to (\R^m,0)$ and $\Phi:(\R^m,0)\to
(\R^m,0)$ such that $ g\circ\phi=\Phi\circ f. $ 
There is also several studies for 
the case that the source space has a boundary. 
In \cite{BG},
map-germs from $2$-dimensional manifolds with boundaries into $\R^2$
are classified, and in \cite{MNa}, map-germs from $3$-dimensional
manifolds with boundaries into $\R^2$ are considered. 
Let $W\subset(\R^m,0)$ be a closed submanifold-germ 
such that $0\in\partial W$ and $\dim W=m$.
We call $f|_W$  a {\it map-germ with boundary,\/}
and we call interior points of $W$ 
{\it interior domain of\/ $f|_W$}.
Since $\partial W$ is an $(m-1)$-dimensional submanifold,
regarding $\partial W=B$,
map-germs from manifolds with boundaries can be treated as a
map-germ $f:(\R^m,0)\to (\R^n,0)$ with a codimension one
oriented submanifold $B\subset (\R^m,0)$. 
We consider $(\R^m,0)$ has an orientation and 
the submanifold $B$ is considered
as the boundary.
We define the interior domain of such map-germ $f$
is the component of $(\R^m,0)\setminus B$ 
such that positively oriented normal vectors of $B$ 
points.
With this terminology,
an equivalent relation for map-germs with boundary 
is the following. 
Let $f,g:(\R^m,0)\to (\R^n,0)$ be map-germs
with codimension one submanifolds $B,B'\subset(\R^n,0)$ 
which contain $0$. Then $f$
and $g$ are {\em ${\cal B}$-equivalent}\/ if there exist an
orientation preserving diffeomorphism $\phi:(\R^m,0)\to (\R^m,0)$
such that $\phi(B)=B'$, and a diffeomorphism $\Phi:(\R^n,0)\to
(\R^n,0)$ satisfies
$$
g\circ\phi=\Phi\circ f.
$$

A map-germ $f:(\R^2,0)\to (\R^3,0)$ is a cuspidal edge if $f$
is $\A$-equivalent to the map-germ $(u,v)\mapsto(u,v^2,v^3)$ at the
origin. We say that $f$ is a {\it cuspidal edge with boundary}\/
$B\subset (\R^2,0)$ if $B$ is a codimension one oriented
submanilfold, that
is, there exists a parametrization $b:(\R,0)\to(\R^2,0)$ to $B$
satisfying $b'(0)\ne(0,0)$.
The domain which lies the left hand side of $b$ 
with respect to the velocity direction
is the interior domain of $f$.

In this note, we will consider differential geometric properties of
cuspidal edges with boundaries. In order to do this, we first
construct a normal form (Proposition \ref{prop:normalform}) of it. 
It can be seen that all the coefficients of the normal form are
differential geometric invariants.
We give geometric meanings of these invariants.
An application of this study is
given by considering flat extensions of flat ruled surfaces with
boundaries. See \cite{mu} for singularities of
the flat ruled surfaces, and see \cite{N} 
for flat extensions of flat ruled surfaces with boundaries.
See \cite{I} for flat extensions from general surfaces.

\section{Normal form of cuspidal edge with boundary}
Now we look for normal form of cuspidal edges with boundary. Let
$f:(\R^2,0)\to (\R^3,0)$ be a cuspidal edge with boundary
$b:(\R,0)\to(\R^2,0)$, $b'(0)\ne(0,0)$. One can take a local
coordinate system $(u,v)$ on $(\R^2,0)$ and an isometry $\Phi$ on
$(\R^3,0)$ satisfying that
\begin{equation}
\label{eq:west3}
\begin{array}{l}
\Phi\circ f (u,v)
 =
\displaystyle
\Big(u,\
\frac{a_{20}}{2}u^2+\frac{a_{30}}{6}u^3+\frac{1}{2}v^2,\
\frac{b_{20}}{2}u^2+\frac{b_{30}}{6}u^3+\frac{b_{12}}{2}uv^2
+\frac{b_{03}}{6}v^3\Big) + h(u,v),
\end{array}
\end{equation}
where $b_{03}\ne0,\ b_{20}\geq0$, and
$$h(u,v)= \big(
0,\
u^4h_1(u),\
u^4h_2(u)+u^2v^2h_3(u)+uv^3h_4(u)+v^4h_5(u,v)
\big),
$$
with\/ $h_1(u),h_2(u),h_3(u),h_4(u),h_5(u,v)$  smooth functions.
See \cite{MS} for details.

Now we consider $b$. Set $b(t)=(b_1(t),b_2(t))$. We divide the
following two cases.
\begin{enumerate}
\item $b_1'(0)\ne0$,
\item $b_1'(0)=0$, $b_2'(0)\ne0$.
\end{enumerate}
In the case $(1)$, one can take $u$ for the parameter of $b$.
Namely, $b$ is parameterized by
\begin{equation}\label{eq:case1}
b(u)=\left(\ep u,\ \sum_{k=1}^3\dfrac{c_{k}}{k!}u^{k}+
u^4c(u)\right)\quad 
(\ep=\pm1).
\end{equation}
In the case $(2)$, one can take $v$ for the parameter of $b$.
Namely, $b$ is parameterized by
\begin{equation}\label{eq:case2}
b(v)=\left(\sum_{k=2}^3\dfrac{d_{k}}{k!}v^{k}+u^4d(u),\ 
\ep v\right)\quad 
(\ep=\pm1).
\end{equation}
In summary, we have the following proposition.
\begin{proposition}\label{prop:normalform}
For any cuspidal edge\/
$f:(\R^2,0)\to(\R^3,0)$ with boundary\/ $b:(\R,0)\to(\R^2,0)$,
there exists a coordinate system on\/ $(\R^2,0)$ and
an isometry\/ $\Phi:(\R^3,0)\to(\R^3,0)$ such that\/
$\Phi\circ f(u,v)$ has the form\/ \eqref{eq:west3}
and\/ $b$ is parameterized by\/ \eqref{eq:case1}
$($respectively, \eqref{eq:case2}$)$
if\/ $b'(0)\not\in \ker df_0$
$($respectively, $b'(0)\in \ker df_0)$.
\end{proposition}
\red{We remark that all coefficients
$c_1,c_2,c_3$, $d_2,d_3$ are geometric
invariants of cuspidal edge with boundary.
Let $f:(\R^2,0)\to(\R^3,0)$ be a cuspidal edge.
Then there exists a unit 
vector field $\nu$ along $f$ satisfying
$\inner{df_p(X)}{\nu(p)}=0$ for any $X\in T_p\R^2$ and 
$p\in (\R^2,0)$, where $\inner{~}{~}$ stands for the Euclidean
inner product of $\R^3$.
We call $\nu$ {\it unit normal vector\/} of $f$.
Moreover, we see that a couple $(f,\nu):(\R^2,0)\to(\R^3\times S^2,(0,\nu(0))$
is an immersion.
Thus a cuspidal edge is a front in the sence of \cite{AGV}. 
See also \cite{front}.}

\begin{figure}[ht]
\centering
\begin{tabular}{cccc}
\includegraphics[width=.22\linewidth,bb=14 14 292 244]
{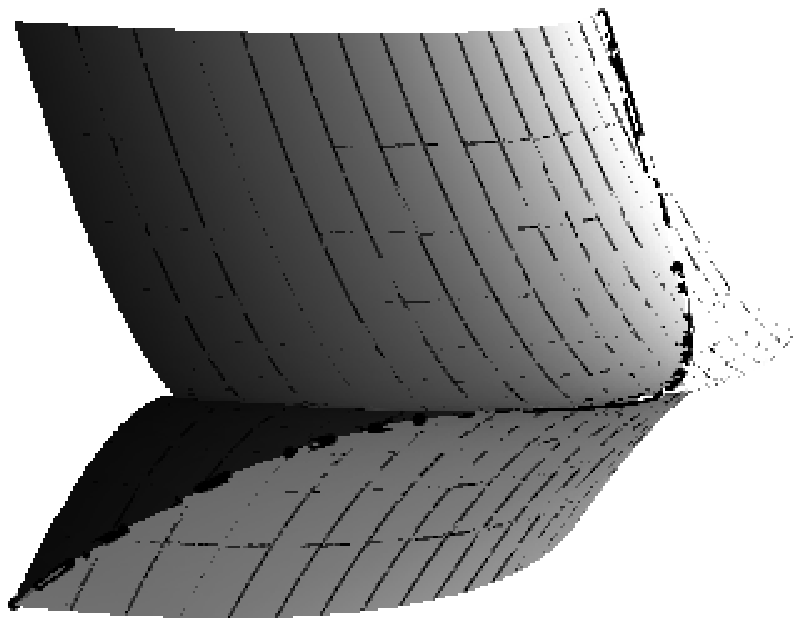} &
\includegraphics[width=.22\linewidth,bb=14 14 262 244]
{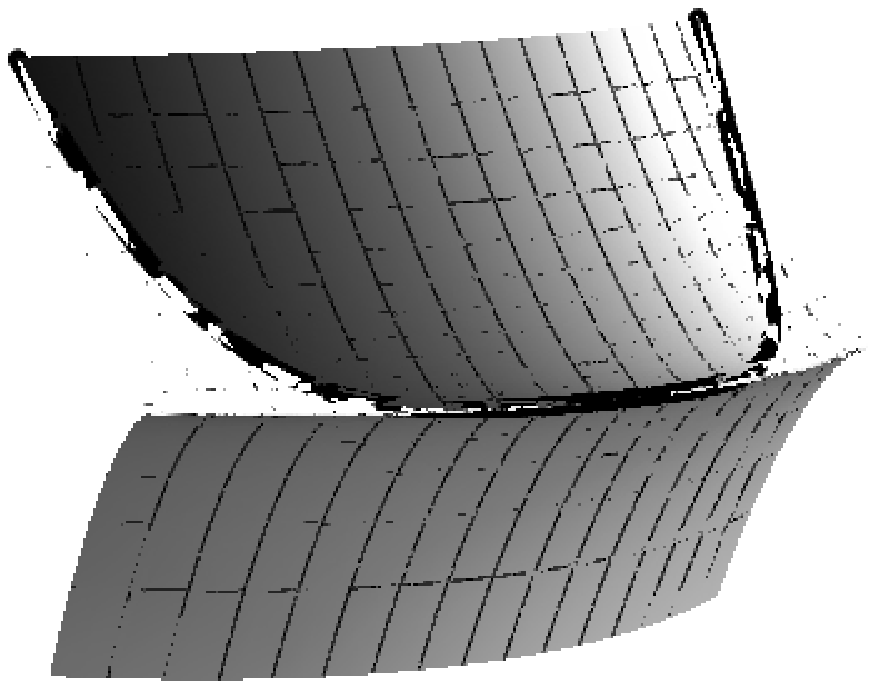} &
\includegraphics[width=.22\linewidth,bb=14 14 321 272]
{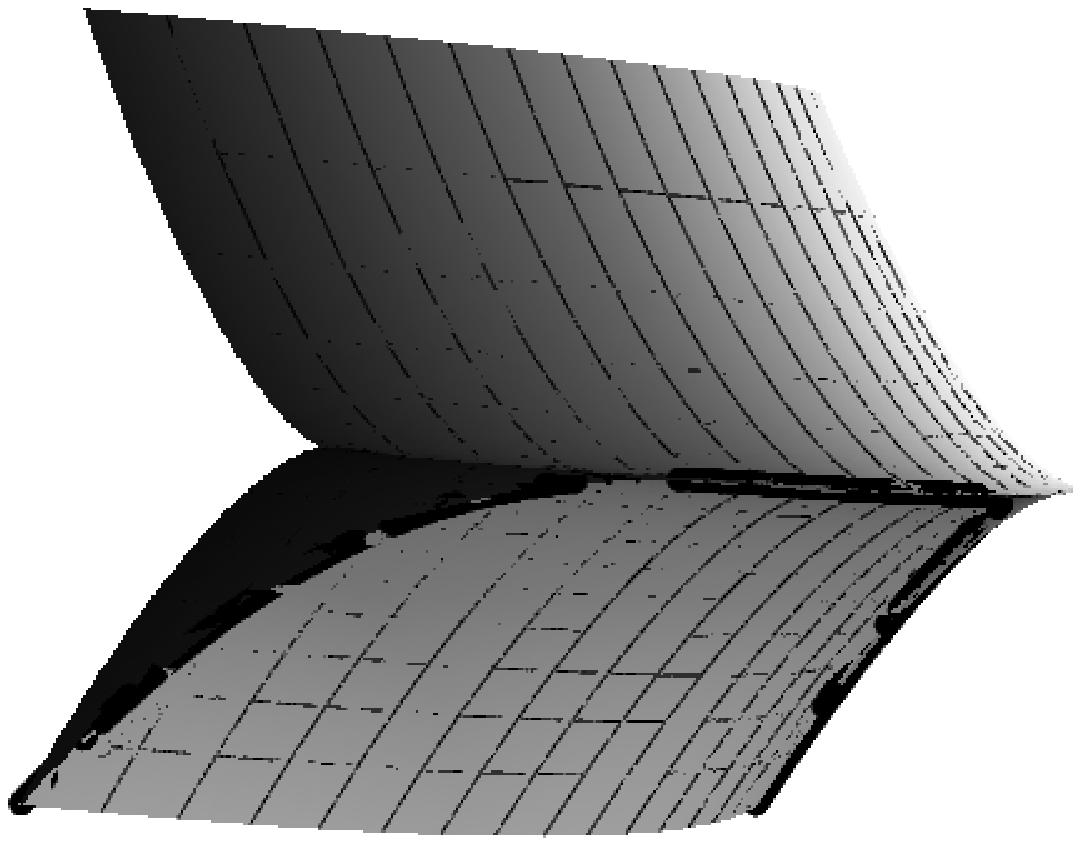} &
\includegraphics[width=.22\linewidth,bb=14 14 295 318]
{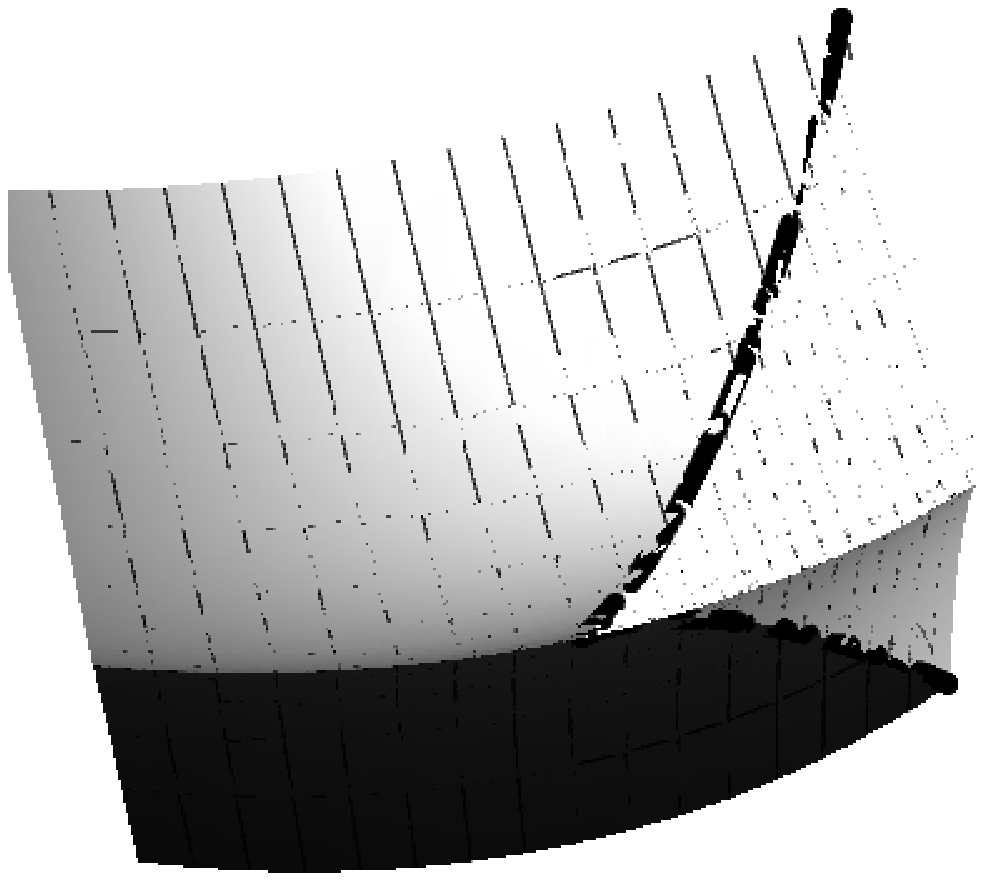}
\end{tabular}
\caption{Cuspidal edges with boundary. The boundaries are drawn by
thick lines, and the exteriors of the surfaces are drawn by thin
colors. Left to right,
$b(t)=(t,t),b(t)=(t,t^2),b(t)=(t,-t^2),b(t)=(t^2,t)$.} \label{fig}
\end{figure}

\section{Differential geometric information}
Several geometric invariants on cuspidal edges are defined and
studied. See \cite{MS,MSUY,front} for details. Coefficients of
\eqref{eq:west3} are invariants and, according to \cite{MS}, it is
known that $a_{20}$ coincides with the singular curvature
$\kappa_s$, $b_{20}$ coincides with the limiting normal curvature
$\kappa_\nu$, $b_{03}$ coincides with the cuspidal curvature 
$\kappa_c$ and $b_{12}$ coincides with the cusp-directional torsion
$\kappa_t$ at the origin.

In what follows, we consider the geometry of the boundary. Let
$f:(\R^2,0)\to (\R^3,0)$ be a cuspidal edge,
$\gamma:(\R,0)\to(\R^2,0)$ a parametrization of its singular set
$S(f)$, and $b:(\R,0)\to(\R^2,0)$ a parametrization of the boundary.
We set $\hat\gamma(t)=f\circ\gamma(t)$ and $\hat b(s)=f\circ b(s)$.

\subsection{The case $(1)$}
We assume that $b'(0)\not\in \ker df_0$ and, by this assumption,
$\hat\gamma=f\circ\gamma$ and $\hat b=f\circ b$ are both regular
curve and they are tangent each other at $0$. 
\red{Hence we have $l\ne0$ such that}
\begin{equation}\label{eq:cebd}
\frac{d}{dt}\hat\gamma\big|_{t=0}=\red{l}\frac{d}{ds}\hat b\big|_{s=0}.
\end{equation}
We take a parametrization of $s$ by $t$ as $s=s(t)$. By the
assumption \eqref{eq:cebd}, $s'(0)=l$. Let $d(t)$ be the curve given
by the difference between $\hat\gamma$ and $\hat b$, that is,
$$
d(t)
=
\hat\gamma(t)-\dfrac{\hat b(s(t))}{\red{l}}.
$$
Then we define the {\em approaching ratio
of boundary to cuspidal edge\/} (or shortly
{\em approaching ratio\/})
by
$$
\alpha= \left|\dfrac{1}{|\hat\gamma'(0)|^3}\det\big(
\hat\gamma'(0),\ d''(0),\ \nu(0,0)\big)\right|^{1/2},\quad
\mathrm{where} \quad {}'=\frac{d}{dt},
$$
\red{where $\nu$ is the unit normal vector of $f$.}
\begin{lemma}\label{lem:indeppara}
The number\/ $\alpha$ does not depend on the choice of the
parameter\/ $t$ and the function\/ $s(t)$.
\end{lemma}
\begin{proof}
Since
$$
\begin{array}{rcl}
d'(t)&=&
\hat\gamma'(t)-\red{\dfrac{1}{l}}\dfrac{d}{ds}\hat b(s(t))s'(t)\\[3mm]
d''(t)&=&
\hat\gamma''(t)-\red{\dfrac{1}{l}}\left(
\dfrac{d^2}{ds^2}\hat b(s(t))(s'(t))^2
-\dfrac{d}{ds}\hat b(s(t))s''(t)\right),
\end{array}
$$
and $(d/ds)\hat b\big|_{s=0}$ is parallel to $\hat\gamma'(0)$,
$s'(0)=\red{l}$, we have
$$
\alpha= \left|\dfrac{1}{|\hat\gamma'(0)|^3}\det\big(
\hat\gamma'(0),\ \hat\gamma''(0)-b_{ss}(0)\red{l},\ 
\nu(0,0)\big)\right|^{1/2}.
$$
Thus $\alpha$ does not depend on $s(t)$. 
We next
assume $t=t(x)$ $(t(0)=0)$ for a parameter $x$, and denote
$(\cdot)_x=(d/dx)(\cdot)$, $(\cdot)_s=(d/ds)(\cdot)$. Then
$$
\begin{array}{CL}
&
\left.\dfrac{\det\Big(\hat\gamma(t(x))_x,\ d(t(x))_{xx},\ \nu_0\Big)}
{|\hat\gamma(t(x))_x|^3}\right|_{x=0}\\[7mm]
=&
\left.\dfrac{\det\Big(\hat\gamma'(t(x))t_x(x),
\hat\gamma''(t(x))(t_x(x))^2
-\hat b_{ss}(s(t(x)))(s'(t(x)))^2(t_x(x))^2\red{l^{-1}},
\nu_0\Big)}
{|\hat\gamma'(t(x))t_x(x)|^3}\right|_{x=0}\\[7mm]
=&
\left.\dfrac{\det\Big(\hat\gamma'(t(x))t_x(x),\
\hat\gamma''(t(x))(t_x(x))^2-\hat b_{ss}(s(t(x)))\red{l}(t_x(x))^2,\
\nu_0\Big)}
{|\hat\gamma'(t(x))t_x(x)|^3}\right|_{x=0}\\[7mm]
=&
\dfrac{\det\Big(\hat\gamma'(t(0)),\
\hat\gamma''(t(0))-\hat b_{ss}(s(t(0)))\red{l},\
\nu_0\Big)}
{|\hat\gamma'(t(0))|^3}
\end{array}
$$
proves the assertion, where $\nu_0=\nu(0,0)$.
\end{proof}
Since the boundary is a curve in $\R^3$, its
curvature $\kappa$ and torsion $\tau$ as a curve in $\R^3$ are
invariants. Moreover, $\hat b$ is a curve on the surface $f$. Thus
the normal curvature $\kappa_{nb}$ and the geodesic curvature
$\kappa_{gb}$ of $b$ are invariants. We have the following
proposition for these invariants.
\begin{proposition}
It hold that
\begin{itemize}
\item $\kappa(0)=\sqrt{b_{20}^2+(c_1^2+a_{20})^2}$,
\item $\kappa'(0)=
\dfrac{
b_{20} (b_{03} c_1^3+3 \ep b_{12} c_1^2+\ep b_{30})
+(c_1^2+a_{20}) (3 c_1 c_2+\ep a_{30})}
{\sqrt{b_{20}^2+(c_1^2+a_{20})^2}}$,
\item $\tau(0)=
\dfrac{(c_1^2+a_{20}) (\ep b_{03} c_1^3+3 b_{12}c_1^2+b_{30})
-b_{20} (3 \ep c_1 c_2+a_{30})}
{b_{20}^2+(c_1^2+a_{20})^2}$,
\item $\kappa_{nb}(0)=b_{20}$,
\item $\kappa_{nb}'(0)=
\dfrac{b_{03} c_1^3}{2}
+2 \ep b_{12}c_1^2
-\dfrac{a_{20}b_{03} c_1}{2}
+\ep b_{30}- \ep a_{20}b_{12}
$,
\item $\kappa_{gb}(0)=-(\ep c_1^2+a_{20})$,
\item $\kappa_{gb}'(0)=
-c_1 \left(\dfrac{\ep b_{03} b_{20}}{2}+3 \ep c_2\right)
-a_{30}-b_{12} b_{20}
$,
\item $\alpha=|c_1|$.
\end{itemize}
\end{proposition}

The invariant $\alpha$ measures the difference of boundary. We can
give a geometric interpretation of $\alpha$ by using the curvature
parabola given by \cite{mn} as follows. Let $f:(\R^2,0)\to(\R^3,0)$
be a map-germ satisfying $\rank df_0=1$, and set
$$
N_0f=\{Y\in \R^3\,;\,\inner{Y}{df_0(X)}=0\text{ for all }
X\in T_0\R^2\},
$$
where we identify $T_0\R^3$ with $\R^3$.
By this identification, $N_0f$ is 
a normal plane of $df_0(X)$ passing through $0$.
The curvature parabola $\Delta_0$ is defined by
$$
\Delta_0=\{
a^2f_{uu}^{\perp}(0)+2abf_{uv}^{\perp}(0) + b^2 f_{vv}^{\perp}(0)
\in N_0f\,;\,a,b\in\R,\
a^2E(0) + 2ab F(0) + b^2 G(0)=1\},
$$
where $E(0) = \langle f_u(0), f_u (0)\rangle$, $F(0) = \langle
f_u(0), f_v (0)\rangle$, $G(0) = \langle f_v(0), f_v (0)\rangle$
and, given $w \in T_0\R^3$, $w^\perp$ is the orthogonal projection
of $w$ at $N_0f$. The curvature parabola is a usual parabola if and
only if $f$ is a cross cap, and otherwise, $\Delta_0$ is a line, a
half-line or a point. In \cite{mn}, the {\it umbilic curvature\/} is
defined by the distance from the origin to $\Delta_0$, if $\Delta_0$
is a half-line, to the line which contains the half-line. If $f$ is
a cuspidal edge, then $\Delta_0$ degenerates in a half-line. In this
case, the umbilic curvature is equal to the limiting normal
curvature defined in \cite{front} up to sign (see also
\cite{mn,MS}). On the other hand, since $\hat{b}$ is tangent to
$\hat\gamma$ at $0$, the principal normal vector $n$ of $\hat b$
lies in  $N_0f$. Let $\ell$ be the line which contains $\Delta_0$.
\begin{lemma}\label{lem:notparallel}
If the limiting normal curvature of the cuspidal edge\/ $f$ is
non zero, then\/ $0\not\in \ell$, and\/
$\ell$ and\/ $n$ are not parallel.
\end{lemma}
\begin{proof}
Without loss of generality, we can take the normal form for $f$
as in \eqref{eq:west3}. Then after some calculation we get that
$$
\Delta_0
=
\{(0,a_{20}+t^2,b_{20})\,;\,t\in\R\},
$$
where the normal plane is $N_0f = \{(0,y,z)\,;\,y,z\in\R\}$. On the
other hand,
$n(0)=(0,c_1^2+a_{20},b_{20})/\sqrt{(c_1^2+a_{20})^2+b_{20}^2}$
which proves the assertion since $b_{20}\neq 0$.
\end{proof}
Let $V$ be the vertex of $\Delta_0$. For instance, for $f$ given as
in \eqref{eq:west3},  $V=(0,a_{20},b_{20})$. By Lemma
\ref{lem:notparallel}, if the limiting normal curvature of $f$ is
non zero,  there exists a intersection point $P$ of lines containing
$n$ and $\ell$.
\begin{proposition}\label{prop:relation}
If the limiting normal
curvature of\/ $f$ is non zero,
then the distance between\/ $V$ and\/ $P$ coincides with\/ $c_1^2$.
\end{proposition}
\begin{proof}
Like as the proof of Lemma \ref{lem:notparallel}, we take the normal
form for $f$. Then  $P=(0,c_1^2+a_{20},b_{20})$, and which proves the
assertion.
\end{proof}
We illustrate the situation in $N_0f$ of Proposition \ref{prop:relation}
in Figure \ref{fig:situ}.
\begin{figure}[ht]
\centering
\includegraphics[width=.7\linewidth]{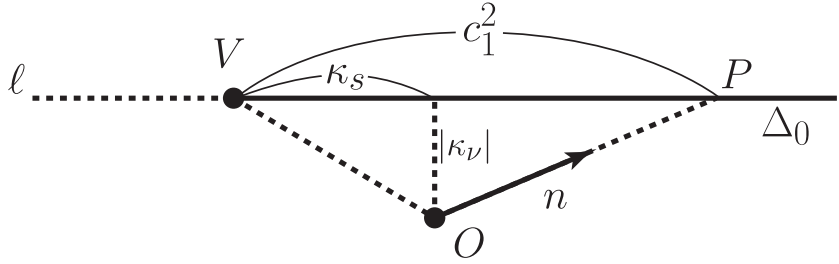}
\caption{Situation of Proposition \ref{prop:relation}.}
\label{fig:situ}
\end{figure}

\subsection{The case $(2)$}
We assume that $b'(0)\in \ker df_0$, and set $\hat b=f\circ b$.
Then we see that $\hat b'(0)=0$ and $\hat b''(0)\ne0$.
Thus we define the {\it angle between boundary and cuspidal edge\/}
by
$$
\beta=
\dfrac{\inner{\hat b''(0)}{\gamma'(0)}}{|\hat b''(0)||\gamma'(0)|}.
$$
One can easily check that $\beta$ does not depend on the choice
of parameters of $b$ and $\gamma$.
If $f$ is given by the normal form \eqref{eq:west3}
with \eqref{eq:case2}, we have $\beta=d_2$.
On the other hand,
since $\hat b$ has a singularity,
the curvature and torsion may diverge.
So we have to prepare curvature and torsion
for singular curve.
See Appendix \ref{sec:kt} for it.
We denote by $\kappa_{sing}$ (respectively, $\tau_{sing}$)
the cuspidal curvature (respectively, the cuspidal torsion)
Then the following proposition holds.
\begin{proposition}
The cuspidal curvature and 
the cuspidal torsion
of\/ $\hat b$ satisfies that
$$
\begin{array}{RCL}
\kappa_{sing}&=&
\dfrac{\sqrt{b_{03}^2 (1+d_2^2)+d_3^2}}
{(1+d_2^2)^{5/4}},\\[4mm]
\tau_{sing}
&=&
\dfrac{
-3 \ep a_{20} b_{03} d_{2}^3
+3b_{20} d_2^2d_3
+6 b_{12}d_2d_3
-h_5(0,0)d_3
+\ep b_{03} d_4
}
{\big(b_{03}^2 (1+d_2^2)+d_3^2\big)^{3/4}}
\sqrt{1+d_2^2}
.
\end{array}
$$
\end{proposition}
\section{Singularities of flat extension of a flat surface}
In this section, as an application of the study on cuspidal edges
with boundary, we consider flat extensions of a flat ruled surface
with boundary. Let $\gamma:I \to \R^3$ be a curve satisfying
$\gamma'(t)\ne0$ for any $t\in I$, where $I$ is an open interval and
$0\in I$. Let $\delta:I\to S^2$ be a curve satisfying
$\delta'(t)\ne0$ for any $t\in I$, where $S^2$ is the unit sphere in
$\R^3$. Then the map $F:I\times (-\ep,\ep)\to \R^3$
\begin{equation}\label{eq:ruled}
F(t,v)=F_{(\gamma,\delta)}(t,v)=\gamma(t)+v \delta(t),
\end{equation}
where $\ep>0$ is called a {\it ruled surface}. It is known that $F$
is flat if and only if $\det(\gamma',\delta,\delta')$ identically
vanishes (See \cite[Proposition 2.2]{izt}, for example.). Since
$\delta\ne0$, one can assume that the parameter $t$ is the
arc-length. Then $\{\delta,\delta',\delta\times\delta'\}$ forms an
orthonormal frame along $\delta$, and
$$
\delta''(t)=-\delta(t)+\kappa_{\delta}(t)\delta(t)\times\delta'(t).
$$
The function $\kappa_{\delta}$ is called the {\it geodesic curvature\/}
of $\delta$, and $\delta$ is determined by $\kappa_\delta$
with an initial condition.
On the other hand, we set
\begin{equation}\label{eq:gamma}
\gamma'(t)=x(t)\delta(t)+y(t)\delta'(t)+z(t)\delta(t)\times\delta'(t).
\end{equation}
Then $\gamma$ is determined by $\{x(t),y(t),z(t)\}$ with an initial
condition. Then $F$ is flat if and only if $z(t)$ identically
vanishes. Moreover, setting $S(F)$ the singular set of $F$, so
$S(F)\cap (I\times[-\ep,\ep])=\emptyset$ if and only if $|y|>\ep$
since $(t,v)$ is a singular point of $F$ if and only if $y(t)+u=0$
as we will see. Thus we set the space of flat ruled surface $FR$ as
$$
FR=\{(x,y,\kappa_\delta)\in
C^\infty(I,\R\times(\R\setminus[-\ep,\ep])\times\R)\}
\times X,
$$
where $X=\{(\delta_0,\delta_1)\in S^2\times S^2\,;\,
\inner{\delta_0}{\delta_1}=0\}$
represents the initial conditions $\delta(0)=\delta_0$
and $\delta'(0)=\delta_1$.

Let us assume that a ruled surface $F=F_{(\gamma,\delta)}$ satisfies
$S(F)\cap (I\times\{0\})=\emptyset$. Then consider extensions of $F$
for $v\in (-M,M)$ $(M>\ep)$ by the same formula \eqref{eq:ruled}. We
call singular points $(t,v)$ of $F$ the {\it birth of
singularities\/} of extension of $F$ if $t$ is a minimal value of
$y(t)$, since $(t,v)$ is a singular point of $F$ if and only if
$y(t)+u=0$.

We have the following result.
\begin{proposition}\label{prop:generic}
Let\/ $I$ be an open interval.
Then the set
$$
\begin{array}{rcl}
{\cal O}
&=&
\{((x,y,\kappa_\delta),(\delta_0,\delta_1))\in FR\,\,;\,
\text{all
birth of singularities of the extensions of\/ }\\
&&\hspace{30mm}F_{(\gamma,\delta)} \text{ are cuspidal edges whose\/
}c_1\text{ vanishes and\/ }c_2\ne0\}
\end{array}
$$
where\/ $\gamma$ is defined by\/ \eqref{eq:gamma}, \/ $\delta$ is
defined by the curvature\/ $\kappa_\delta$ with the initial
condition\/ $\delta_0,\delta_1$ being open and dense in\/ $FR$ with
respect to the Whitney\/ $C^\infty$ topology, and $c_0, c_1$
are given by \eqref{eq:case1}.
\end{proposition}
To prove this proposition, we show the following lemma.
\begin{lemma}\label{lem:ruledsing}
For a flat ruled surface $F$ as in\/ \eqref{eq:ruled},
\begin{itemize}
\item $(t,v)$ is a singular point of $F$ if and only if\/ $y(t)+u=0$.
\item $F$ is a cuspidal edge at $(t,v)\in S(F)$
if and only if\/ $y'(t)-x(t)\ne0$, $\kappa_{\delta}(t)\ne0$.
\end{itemize}
\end{lemma}
\begin{proof}
Since $F'=\gamma'+u\delta'=x+(y+u)\delta'$ and $F_u=\delta$,
where we omit $(t)$ and $'=\partial/\partial t$,
$(\cdot)_u=\partial/\partial u$,
we see the first assertion.
Moreover, we see that
$\ker dF_{(t,v)}=\langle\partial t-x\partial u\rangle_{\R}$
for $(t,v)\in S(F)$,
and $\delta\times\delta'$ gives a unit normal vector of $F$.
Set $\eta=\partial t-x\partial u$.
Thus we see that
$\eta(\delta\times\delta')=\kappa\delta$,
and
$\eta(y+u)=y'-x$.
By the well-known criteria for cuspidal edge
(\cite[Corollary 2.5]{suy3}, see also \cite[Proposition 1.3]{krsuy}),
we see the second assertion.
\end{proof}
\begin{proof}[Proof of Proposition\/ {\rm \ref{prop:generic}}.]
We define subsets of the $2$-jet space
$J^2(I,\R\times(\R\setminus[-\ep,\ep])\times\R)$
as follows:
\begin{equation}
\begin{array}{RCL}
C_1&=&\{j^2(x,y,\kappa_{\delta})(t,v)\,;\,\kappa_\delta(t)=0\}\\
C_2&=&\{j^2(x,y,\kappa_{\delta})(t,v)\,;\,y'(t)-x(t)=0\}\\
C_3&=&\{j^2(x,y,\kappa_{\delta})(t,v)\,;\,y'(t)=0\}\\
C_4&=&\{j^2(x,y,\kappa_{\delta})(t,v)\,;\,y''(t)=0\}
\end{array}
\end{equation}
Since a coordinate system of
$J^2(I,\R\times(\R\setminus[-\ep,\ep])\times\R)$ is given by
$(t,x,y,\kappa_{\delta},x',y', \kappa_{\delta}',$ $x'',$ $y'',$
$\kappa_{\delta}'')$, we see that these subsets are closed
submanifolds with codimension $1$, and $C_i\cap C_3$ $(i=1,2,4)$ are
closed submanifolds with codimension $2$. By the Thom jet
transversality theorem, the set
$$
\begin{array}{l}
{\cal O}'=
\{
((x,y,\kappa_\delta),(\delta_0,\delta_1))\in FR\,\,;\,
j^2(x,y,\kappa_\delta):I\mapsto
J^2(I,\R\times(\R\setminus[-\ep,\ep])\times\R)\\
\hspace{35mm} \text{is transverse to }
C_1,C_2,C_3,C_4\text{ and } C_i\cap C_3\ (i=1,2,4)\}
\end{array}
$$
is a residual subset of $FR$. Let
$((x,y,\kappa_\delta),(\delta_0,\delta_1))\in {\cal O}'$ and assume
that $(t_0,v_0)$ is a birth of singularity of $F$. Since $(t_0,v_0)$
is a birth of singularity, and $S(F)=\{y(t_0)-u_0=0\}$, we see
$y'(t_0)=0$. Since $y'(t_0)=0$ and $(x,y,\kappa_\delta)\in {\cal
O}'$, $F$ at $(t_0,v_0)$ is a cuspidal edge by Lemma
\ref{lem:ruledsing}. Moreover, we have $y''(t_0)\ne0$. This implies
that the contact of $S(F)$ and the $t$-curve $\{(t,v)\,;\,v=v_0\}$
is of second degree. On the other hand, the condition $c_1=0$ and
$c_2\ne0$ as in \eqref{eq:case1} implies that the contact of $S(f)$
(the $u$-axis) and $b$ is of second degree. Since the degrees of
contact of two curves do not depend on the diffeomorphism, the
cuspidal edge $F$ at $(t_0,v_0)$ has the property $c_1=0$ and
$c_2\ne0$. This proves the assertion.
\end{proof}
We remark that singularities of 
flat surfaces with boundaries are studied in \cite{mu},
and the flat extensions of flat surfaces
are studied in \cite{N}. 
Flat extensions of generic surfaces with boundaries
are studied in \cite{I}. 
In \cite{IO}, flat ruled surfaces approximating regular surfaces
are studied.

\appendix
\section{Curvature and torsion of space curves with singularities}
\label{sec:kt} In the case (2), the image of the boundary of a
cuspidal edge with boundary has a singularity. Thus we need
differential geometry of space curves with singularities. In this
appendix we give curvature and torsion for space curves with
singularities. It should be mentioned that the discussions here are
quite analogies of the study for the case of plane curves given by
Shiba and Umehara \cite{su}, and we follow their discussions in the
following.

Let $\gamma:(\R,0)\to(\R^3,0)$ be a curve and assume that
$\gamma'(0)=(0,0,0)$. We say that $0$ is called $A$-type if
$\gamma''(0)\ne(0,0,0)$, and $0$ is called $(2,3)$-type if
$\gamma''(0)\times\gamma'''(0)\ne(0,0,0)$. 
Let $0$ be a $A$-type singular
point of $\gamma$, then we define
$$
\kappa_{sing}
=
\dfrac{|\gamma''(0)\times \gamma'''(0)|}
{|\gamma''(0)|^{5/2}}.
$$
We call $\kappa_{sing}$ the {\it cuspidal curvature\/} of $\gamma$.
This definition is analogous to the cuspidal curvature
for $(2,3)$-cusp of plane curve introduced in \cite{u}.
See \cite{suyo} for detail.
Moreover, let $0$ be a $(2,3)$-type singular point of $\gamma$,
then we define
$$
\tau_{sing}
=
\dfrac{\sqrt{|\gamma''(0)|}\det(\gamma''(0),\gamma'''(0),\gamma''''(0))}
{|\gamma''(0)\times \gamma'''(0)|^2}.
$$
We call $\tau_{sing}$ the {\it cuspidal torsion\/} of $\gamma$.
By a direct calculation, one can show
that $\kappa_{sing}$ and $\tau_{sing}$
do not depend on the choice of parameter.
Furthermore, we have the following.
Let $s_g$ be the arc-length function $s_g(t)=\int_0^t|\gamma'(t)|\,dt$.
\begin{fact}{\rm (\cite[Theorem 1.1, Lemma 2.1]{su})}
The functions
$$\sgn(t)\sqrt{|s_g(t)|}
\quad\text{and}\quad
\sqrt{|s_g(t)|}\kappa(t)
$$
are\/ $C^\infty$-differentiable,
and
$$
\lim_{t\to0}\sqrt{|s_g(t)|}\kappa(t)=
\frac{1}{2\sqrt{2}}\kappa_{sing}.
$$
\end{fact}
By this fact, $\sgn(t)\sqrt{|s_g(t)|}$
can be taken as a local
coordinate of the curve $\gamma$ at $t=0$.
It is called {\it half-arclength\/} parameter.
We have an analogous claim for the torsion.
\begin{proposition}
The function\/
$\sgn(t)\sqrt{|s_g(t)|}\tau(t)$
is\/ $C^\infty$ differentiable,
and
$$
\lim_{t\to0}\sgn(t)\sqrt{|s_g|}\tau(t)=\frac{2}{3\sqrt{2}}\tau_{sing}.
$$
\end{proposition}
\begin{proof}
By L'H\^ospital's rule,
we see
$$
\lim_{t\to0}
\dfrac{|\gamma'\times\gamma''|^2}{t^4}=
\dfrac{6|\gamma''(0)\times\gamma'''(0)|}{4!},\quad
\lim_{t\to0}
\dfrac{\det(\gamma',\gamma'',\gamma''')}{t^3}=
\dfrac{\det(\gamma''(0),\gamma'''(0),\gamma''''(0))}{3!}.
$$
Thus these two functions are $C^\infty$-differentiable at $t=0$.
Moreover,
$$
\lim_{t\to0}t\tau(t)
=
\lim_{t\to0}
\dfrac{\det(\gamma',\gamma'',\gamma''')}{t^3}
\dfrac{t^4}{|\gamma'\times\gamma''|^2}
=
\dfrac{\det(\gamma''(0),\gamma'''(0),\gamma''''(0))}{3!}
\dfrac{4!}{|\gamma''(0)\times\gamma'''(0)|^2}
$$
shows that $t\tau(t)$ is $C^\infty$-differentiable. On the other
hand, by L'H\^ospital's rule, we have
\begin{equation}\label{eq:halfarc}
\lim_{t\to0}
\left|\dfrac{s_g(t)}{t^2}\right|
=
\lim_{t\to0}
\dfrac{|\gamma'(t)|}{|2t|}
=
\dfrac{|\gamma''(0)|}{2}.
\end{equation}
Thus
$$
\lim_{t\to0}
\dfrac{\sqrt{|s_g(t)|}}{|t|}
=
\dfrac{\sqrt{|\gamma''(0)|}}{\sqrt{2}}.
$$
Hence
$$
\begin{array}{L}
\lim_{t\to0}
\sgn(t)|t|\dfrac{\sqrt{|\gamma''(0)|}}{\sqrt{2}}
\tau
=
\dfrac{\sqrt{|\gamma''(0)|}}{\sqrt{2}}
\lim_{t\to0}
\dfrac{\det(\gamma',\gamma'',\gamma''')}{t^3}
\dfrac{t^4}{|\gamma'\times\gamma''|^2}\\[5mm]
\hspace{40mm}
=
\dfrac{2}{3\sqrt{2}}\,
\dfrac{\sqrt{|\gamma''(0)|}\det(\gamma''(0),\gamma'''(0),\gamma''''(0))}
{|\gamma''(0)\times\gamma'''(0)|^2}
\end{array}$$
which shows the assertion.
\end{proof}
We remark that this proof is analogous to that of \cite[Lemma 2.1]{su}.
Thus $\kappa_{sing}$ (respectively, $\tau_{sing}$) is a geometric
invariant of $A$-type (respectively, $(2,3)$-type) singular
space curve, and it can be regarded
as a natural limit of usual curvature (respectively, torsion).
We also remark that an $A$-type space curve-germ $\gamma:(\R,0)\to(\R^3,0)$
at $0$ is $(2,3)$-type if and only if $\kappa_{sing}\ne0$.
By \eqref{eq:halfarc},
a parametrization $t$ of the $A$-type space curve-germ $\gamma$
is the half-arclength parameter if and only if $|\gamma'(t)|=2|t|$
(see \cite[Remark 2.2]{su}).
We have the following proposition.
\begin{proposition}
Let\/ $\alpha,\beta:(\R,0)\to\R$ be\/ $C^\infty$-functions
satisfying\/ $\alpha>0$.
Then there exists a unique\/ $(2,3)$-type curve-germ\/
$
\gamma:(\R,0)\to(\R^3,0)
$
up to orientation preserving isometric
transformations in\/ $\R^3$ such that
\begin{equation}\label{eq:kt}
\sqrt{|s_g(t)|}\kappa(t)=\alpha(t)\quad\text{and}\quad
\sqrt{|s_g(t)|}\tau(t)=\beta(t)
\end{equation}
and\/ $t$ is the half-arclength parameter.
\end{proposition}
\begin{proof}
Let us consider an ordinary differential equation
$$
A'(t)=2A(t)
\pmt{0&-\alpha(t)&0\\ \alpha(t)&0&-\beta(t)\\0&\beta(t)&0}.
$$
Then we see that $A(t)$ is an orthonormal matrix under an initial
condition and $A(0)$ is the identity matrix. Set
$A(t)=(e(t),n(t),b(t))$ and set $\gamma(t)=2\int_0^t te(t)\,dt$.
Then $|\gamma'(t)|=2|t|$ and which shows that $t$ is the half-arclength
parameter. One can easily see that $\gamma(t)$ satisfies
\eqref{eq:kt}.
\end{proof}
We remark that this proof is analogous to that of \cite[Theorem 1.1]{su}.

For a space curve-germ $\gamma$ of $A$-type,
one can easily see that there exist a parameter $t$ and
an isometry $A$ such that
\begin{equation}\label{eq:normalatype}
A\circ \gamma(t)=
\left(
\dfrac{t^2}{2},
\sum_{i=3}^l\dfrac{1}{i!}{\gamma_{2i}}t^i,
\sum_{i=4}^l\dfrac{1}{i!}{\gamma_{3i}}t^i
\right)+(0,O(l+1),O(l+1)),
\end{equation}
where
$O(l+1)$ stands for the terms whose degrees are greater than $l+1$,
and $\gamma_{ji}\in\R$ $(j=2,3,\ i=2,\ldots,l)$.
If $\gamma$ is of $(2,3)$-type, then $\gamma_{23}\ne0$, and we see that
$$
\kappa_{sing}=\dfrac{|\gamma_{23}|}{2\sqrt{2}},\quad
\tau_{sing}=
\dfrac{\gamma_{34}}{\gamma_{23}}.
$$
We set 
$$
\kappa_{sing}'=
\left.
\dfrac{d}{dt}\left(\sqrt{|s_g(t)|}\kappa(t)\right)\right|_{t=0}.
$$
Then $\kappa_{sing}'=(\gamma_{23}+4\gamma_{24})/(12\sqrt{2}|\gamma_{23}|)$.
Hence we would like to say that $\kappa_{sing},\kappa_{sing}',\tau_{sing}$
are all invariants of $(2,3)$-type singular
space curve up to fourth degree.
However, it is not easy to compute the differentiation
of $\sqrt{|s_g(t)|}\kappa(t)$ for a given curve.
Thus we set
$$
\sigma_{sing}=
\dfrac{
\left(
\inner{\gamma''(0)\times\gamma'''(0)}{\gamma''(0)\times\gamma^{(4)}(0)}
-2\dfrac{|\gamma''(0)\times\gamma'''(0)|^2\inner{\gamma''(0)}{\gamma'''(0)}}
{\inner{\gamma''(0)}{\gamma''(0)}}\right)}
{\inner{\gamma''(0)}{\gamma''(0)}^{11/4}}.
$$
Then this is independent of the choice of the parameter,
and
$$
\sigma_{sing}=\gamma_{23}(\gamma_{24}-2\gamma_{23})
$$
holds for $\gamma$ of the form \eqref{eq:normalatype}. 
Thus invariants $\{\kappa_{sing},\sigma_{sing},\tau_{sing}\}$ 
can be used instead of $\{\kappa_{sing},\kappa_{sing}',\tau_{sing}\}$
for $(2,3)$-type singular
space curve up to fourth degrees.


\medskip
{\small
\begin{flushright}
\begin{tabular}{l}
\begin{tabular}{l}
Departamento de Matem{\'a}tica,\\
Instituto de Bioci\^{e}ncias, Letras e Ci\^{e}ncias Exatas,\\
 UNESP - Univ Estadual Paulista,\\
  C\^{a}mpus de S\~{a}o Jos\'{e} do Rio Preto, SP, Brazil\\
  E-mail: {\tt lmartinsO\!\!\!aibilce.unesp.br}
\end{tabular}
\\
\\
\begin{tabular}{l}
Department of Mathematics,\\
Graduate School of Science, \\
Kobe University, \\
Rokkodai 1-1, Nada, Kobe 657-8501, Japan\\
  E-mail: {\tt sajiO\!\!\!amath.kobe-u.ac.jp}
\end{tabular}
\end{tabular}
\end{flushright}}
\end{document}